\documentclass{amsart}
\usepackage{amsfonts,amssymb,amscd,amsmath,enumerate,verbatim,calc}
\usepackage[all]{xy}

\newcommand{\m}{\mathfrak{m} }
\newcommand{\F}{\mathfrak{F} }

\newcommand{\A}{\mathfrak{A} }
\newcommand{\R}{\mathcal{R} }
\newcommand{\Sc}{\mathcal{S} }
\newcommand{\Z}{\mathbb{Z} }

\newcommand{\rt}{\rightarrow}

\newcommand{\ov}{\overline}

\newcommand{\wt}{\widetilde }

\newcommand{\image}{\operatorname{image}}

\newcommand{\smod}{\operatorname{mod}}
\newcommand{\grade}{\operatorname{grade}}

\newcommand{\Ass}{\operatorname{Ass}}
\newcommand{\ann}{\operatorname{ann}}

\newcommand{\gmod}{\operatorname{ ^*mod}}
\newcommand{\Lmod}{\operatorname{ ^*Mod}}
\newcommand{\coker}{\operatorname{coker}}
\newcommand{\Hom}{\operatorname{Hom}}
\newcommand{\Ext}{\operatorname{Ext}}
\newcommand{\Tor}{\operatorname{Tor}}

\theoremstyle{plain}

\newtheorem{theorem}{Theorem}[section]
\newtheorem{corollary}[theorem]{Corollary}
\newtheorem{lemma}[theorem]{Lemma}
\newtheorem{proposition}[theorem]{Proposition}

\theoremstyle{definition}
\newtheorem{definition}[theorem]{Definition}
\newtheorem{s}[theorem]{}
\newtheorem{remark}[theorem]{Remark}

\theoremstyle{remark}

\begin{document}

\title[Asymptotic primes]{A convenient category to study asymptotic primes and related questions}
\author{Tony~J.~Puthenpurakal}
\date{\today}
\address{Department of Mathematics, IIT Bombay, Powai, Mumbai 400 076}

\email{tputhen@math.iitb.ac.in}

 \begin{abstract}
Let $A$ be a Noetherian ring and let $\mathcal{R} = \bigoplus_{n \geq 0}\R_n$ be a standard graded ring with $\mathcal{R}_0 = A$. We define a category $\mathfrak{A}(\mathcal{R})$ of graded $\mathcal{R}$-modules (not necessarily finitely generated) with the following properties: if $X  = \bigoplus_{n \in \Z} X_n \in \mathfrak{A}(\mathcal{R}) $ then
\begin{enumerate}
  \item  $X_i$ is finitely generated $A$-module for all $i \in \Z$ and $X_i = 0$ for $i \ll 0$.
  \item There exists $n_0$ such that $\Ass_A X_n = \Ass_A X_{n_0}$ for all $n \geq n_0$.
  \item If $X_n$ has finite length  as an $A$-module for all $n$ then there exists $P_X(z) \in \mathbb{Q}[z]$ such that $P_X(n) = \ell_A(X_n)$ for all $n \gg 0$.
  \item If $F$ is a coherent functor on the category of finitely generated $A$-modules then $F(X) = \bigoplus_{n \in \Z} F(X_n) \in \mathfrak{A}(\mathcal{R})$.
  \item For an ideal $J$ in $A$, there exists $c_J^X$ such that \\ $\grade(J, X_n) = \grade(J, X_{c_J^X})$ for all $n \geq c_J^X$.
  \end{enumerate}
We give a unified proof of several results in theory of associate primes and related areas.
\end{abstract}
 \maketitle
\section{introduction}
The theory of asymptotic prime divisors was initiated by Ratliff \cite{R}. Brodmann  \cite[p.\ 16]{Br}  showed that when $M$ is a finite module over a Noetherian commutative ring $A$ and $I$ is an ideal of $A$, then the sequences of sets of associated prime ideals $\Ass_A(I^n M/I^{n+1}M)$ and $\Ass_A(M/I^nM)$, $n = 1, 2, \ldots$ are ultimately constant. Brodmann’s result has many applications and has been generalized in various forms. In \cite[Proposition 2]{ME}, McAdam and Eakin proved that if $R$ is a standard graded
Noetherian ring with $R_0 = A$ and $M$ is a finitely generated graded $A$-module, then the set of
primes $\Ass_A(M_n)$ is stable for all large $n$. They also showed that Brodmann’s result followed
from their result concerning the stability of primes associated to homogeneous components of
graded modules.
Later Brodmann, \cite[Theorem 2(i) and  12(i)]{Br2} proved that if $J$ is an ideal in $A$ then $\grade(J, I^nM/I^{n+1}M)$ and $\grade(J, M/I^{n}M) $ stabilize. He used this result to give a refinement of an inequality due to Burch, \cite[Theorem 2(ii)]{Br2}.  For a nice introduction to the theory of asymptotic primes see \cite{Mc}.

In \cite[Theorem 1]{MS} the authors prove that for each $i \geq 0$ the sets \\ $\Ass_A \Tor^A_i(M, A/I^n)$ and $\Tor^A_i(M, I^n/I^{n+1})$ stabilize. Later in \cite[3.5]{KW} the authors proved that for finitely generated $A$-modules $M, N$ the sets $\Ass_A \Tor^A_i(M, N/I^nN)$ and $\Ass_A \Ext_A^i(M, N/I^nN)$ are stable.

In \cite[Theorem 2]{KV} it was shown that if $A$ is local and $\ell(M\otimes_A N)$ is finite then the functions $n \rt \ell(\Tor^A_i(M, N/I^nN))$ and $n \rt \ell(\Ext^i_A(M, N/I^nN))$ are of polynomial type. In \cite[Corollary 4]{T}
the result was generalized by just assuming that $\Tor^A_i(M, N/I^nN)$ is of finite length for all $n$ then the function $n \rt \ell(\Tor_i^A(M, N/I^nN))$  is of polynomial type (similar result holds for \\  $\Ext^i_A(M, N/I^nN)$).

Recently in a  far-reaching  generalization T.Se, \cite[2.4, 2.7, 2.13]{Se},   showed that if $F$ is a coherent covariant $A$-linear functor from the category of finitely generated $A$-modules to itself the sets $\Ass_A(F(I^nM/I^{n+1}M))$, $\Ass_A(F(M/I^nM))$,and the values  $\grade(J,F(I^nM/I^{n+1}M))$ and $\grade(J, F(M/I^nM))$ become independent of $n$ for large $n$. We note that if $N$ is a finitely generated $A$-module then the functors $\Tor^A_i(N,-)$ and $\Ext^i_A(N,-) $ are coherent for all $i\geq 0$; see \cite[2.4, 2.5]{H}. In \cite[3.3]{BM} it is proved that if $F(M/I^nM)$ has finite length  for all $n$ then
the function $n \rt \ell(F(M/I^nM))$ is of polynomial type.

In \cite[2.1]{KN} the authors proved that if $\R $ is a standard graded subring of a polynomial ring $\Sc = A[X_1, \ldots, X_m]$ with $\R_0 = A$ then the set $\Ass_A \Sc_n/\R_n$ is stable for all $n \gg 0$.
In \cite[1.1]{HF}, this result was generalized  as follows: Let $\R, \Sc$ be standard graded rings with $\R_0 = \Sc_0 = A$ and we have an inclusion of graded rings $\R \subseteq \Sc$. Assume $M$ is a finitely generated graded $\R$-module and $N$ is a finitely generated $\Sc$ module. If $M \subseteq N$ as graded $\R$-modules then $\Ass_A N_n/M_n$ is stable for $n \gg 0$. Furthermore if $J$ is an ideal of $A$ then $\grade(J, N_n/M_n)$ is constant for $n \gg 0$, see \cite[3.1]{HF}.

\s \label{question} In view of \cite{Se} we may ask that if $ F$ is a coherent functor  and $M, N$ as above then
\begin{enumerate}
\item
is
$\Ass_A F(N_n/M_n)$ stable for $n \gg 0$ ?
\item
if $J$ is an ideal in $A$ then is $\grade(J, F(N_n/M_n))$ is constant for $n \gg 0$ ?
\item
if $F(N_n/M_n)$ has finite length for all $n$ then is
the function \\ $n \rt \ell(F(N_n/M_n))$  of polynomial type?
\end{enumerate}
The techniques in \cite{Se}, \cite{BM} and \cite{HF} do not generalize to deal with this case. In this paper we show all the above questions have a positive answer.

\s \label{main-intro} Our technique to prove the above assertions is categorical. Let $\R$ be a standard graded ring with $\R_0 = A$. We concoct a full subcategory $\A(\R)$ of  $\Lmod(\R)$, the category of all graded $\R$-modules, with four properties:
if $X  = \bigoplus_{n \in \Z}X_n\in \A(\R)$  then
\begin{enumerate}
  \item  $X_n$ is finitely generated $A$-module for all $n$ and $X_n = 0$ for $n \ll 0$.
  \item  $\Ass_A X_n$ is stable for all $n \gg 0$
  \item If $\ell(X_n)$ is finite for all $n$ then the function $n \rt \ell (X_n)$ is of polynomial type.
  \item If  $F$ is a coherent functor on the category of finitely generated $A$-modules then $F(X) = \bigoplus_{n \in \Z} F(X_n)\in \A(\R)$.
\end{enumerate}
A formal consequence of the above four properties is that if $J$ is an ideal of $A$ then there exists $c_J^X$ such that  $\grade(J, X_n) = \grade(J, X_{c_J^X})$ for all $n \geq c_J^X$.
As a consequence we give an affirmative answer to all the questions in \ref{question}.

\s \label{cat} We now describe the category that we are interested in.\\
Let $A$ be a Noetherian ring and let $\R = \bigoplus_{n \geq 0}\R_n$ be a standard graded ring with $\R_0 = A$. Let $\gmod(\R)$ denote the category of all finitely generated graded $\R$-modules and let $\Lmod(\R)$ denote the  category of all graded $\R$-modules. Define $\A(\R)$ a full subcategory of $\Lmod(\R)$  iteratively as follows:

Define $\A(\R)(0) = \gmod(\R)$.

A graded $\R$-module $X \in \A(\R)(1)$ if there exists a standard graded $A$-algebra $S$ with $R\subseteq S$ (as graded rings) with $R_0 = S_0$ and there exists $M \in \gmod(R)$ and $N \in \gmod(S)$ with $M \subseteq N$ (as graded $\R$-modules) and $X = N/M$. Clearly $\A(\R)(0) \subseteq  \A(\R)(1)$.

A graded $\R$-module $Y \in \A(\R)(2)$ if there exists  $X \in \A(\R)(1)$ and $D \in \A(\R)(0)$ such that there exists an exact sequence of graded $\R$-modules
\[
0 \rt X \rt Y \rt D \rt 0.
\]
Clearly $\A(\R)(1) \subseteq \A(\R)(2)$. We define $\A(\R) = \A(\R)(2)$.

The main theorem of this paper is
\begin{theorem}\label{main}
Let $A$ be a Noetherian ring and let $\R = \bigoplus_{n \geq 0}\R_n$ be a standard graded ring with $\R_0 = A$. Let $X  = \bigoplus_{n \in \Z} X_n\in \A(\R)$. Then
\begin{enumerate}[\rm (1)]
  \item  $X_i$ is finitely generated $A$-module for all $i \in \Z$ and $X_i = 0$ for $i \ll 0$.
  \item There exists $n_0$ such that $\Ass_A X_n = \Ass_A X_{n_0}$ for all $n \geq n_0$.
   \item If $A$ is local and $X_n$ has finite length  as an $A$-module for all $n$ then there exists $P_X(z) \in \mathbb{Q}[z]$ such that $P_X(n) = \ell_A(X_n)$ for all $n \gg 0$.
  \item If $F$ is a coherent functor on the category of finitely generated $A$-modules then $F(X) = \bigoplus_{n \in \Z} F(X_n)\in \A(\R)$.
\item For an ideal $J$ in $A$, there exists $c_J$ such that \\ $\grade(J, X_n) = \grade(J, X_{c_J})$ for all $n \geq c_J$.
\end{enumerate}
\end{theorem}

We now describe in brief the contents of this paper. In section two we discuss some preparatory results on $\A(\R)$ that we need. In section three we recall the notion of coherent functors and prove that if $X = \bigoplus_{n \in \Z}X_n \in \A(\R)$ and $F$ is a coherent functor of finitely generated $A$-modules then $F(X) = \bigoplus_{n\in \Z} F(X_n) \in \A(\R)$. In section four we recall the notion of quasi-finite modules and prove some results that we need. In section five we prove an analogue of a result due to Amao that we need. In section six we prove Theorem \ref{main} and give an affirmative answer to the questions raised in \ref{question}.
\section{Some preparatory results on $\A(\R)$ }
In this section we prove some preparatory results on $\A(\R)$ that we need. \\ Throughout $A$ is a Noetherian ring and $R$ is  a standard graded algebra with $R_0 = A$.
\begin{lemma}
\label{closed}  Let $0 \rt Y \xrightarrow{i} Z \rt D \rt 0$  be an exact sequence of graded $\R$-modules with $Y \in \A(\R)$ and $D \in \gmod(\R)$.
Then $Z \in \A(\R)$.
\end{lemma}
\begin{proof}
As $Y$ is in $\A(\R)$ there exists an exact sequence of $\R$-modules $0 \rt X \rt Y \xrightarrow{\pi} V \rt 0$ with $X \in \A(\R)(1)$ and $D \in \gmod(\R)$.
Let $C$ be the pushout of $i$ and $\pi$. Thus we have a commutative diagram with exact rows
\[
  \xymatrix
{
 0
 \ar@{->}[r]
  & Y
\ar@{->}[r]^{i}
\ar@{->}[d]^{\pi}
 & Z
\ar@{->}[r]^{p}
\ar@{->}[d]^{g}
& D
\ar@{->}[r]
\ar@{->}[d]^{j}
&0
\\
 0
 \ar@{->}[r]
  &V
\ar@{->}[r]^{i^\prime}
 & C
\ar@{->}[r]^{p^\prime}
& D
    \ar@{->}[r]
    &0
\
 }
\]
Here $j$ is the identity map on $D$. We note that $C \in \gmod(\R)$. Furthermore $g$ is surjective with $\ker g \cong X$. It follows that
$Z \in \A(\R)$.
\end{proof}
The category $\A(\R)$ is NOT abelian. However certain cokernels exist as we show in our next result.
\begin{proposition}\label{q-fg}
Let $D \in \gmod(\R)$ and let $X \in \A(\R)$. If $D \subseteq X$ then $X/D \in \A(\R)$.
\end{proposition}
\begin{proof}
  We show step-by step that if $X \in \A(\R)(i)$ then $X/D \in \A(\R)(i)$. We have nothing to show when $i = 0$ as $\A(\R)(0) = \gmod(\R)$. Let $X \in \A(\R)(1)$. Then there exists a standard graded algebra $S$ with $S_0 = A$ with an inclusion of graded rings $\R \subseteq S$ and $M \in \gmod(\R)$ and $N \in \gmod(S)$ with $M \subseteq N$ and $N/M = X$. We have a commutative diagram
  \[
  \xymatrix
{
 0
 \ar@{->}[r]
  & M
\ar@{->}[r]
\ar@{->}[d]^{\phi}
 & N
\ar@{->}[r]
\ar@{->}[d]^{1_N}
& X
\ar@{->}[r]
\ar@{->}[d]^{\pi}
&0
\\
 0
 \ar@{->}[r]
  &L
\ar@{->}[r]
 & N
\ar@{->}[r]
& X/D
    \ar@{->}[r]
    &0
\
 }
\]
We note $\phi$ is injective and $\coker \phi = D$. So $L \in \gmod(\R)$. It follows that $X/D \in \A(\R)(1)$.

Now let $X \in \A(\R)(2)$. Then there exists an exact sequence
$$0 \rt Y \xrightarrow{\phi} X \xrightarrow{f} V \rt 0 \quad \text{with} \ Y \in \A(\R)(1), \text{and} \ V \in \gmod(\R).$$
Let $V' = f(D)$. Then we have a surjective map $\ov{f} \colon X/D \rt V/V'$. Let $\pi \colon X \rt X/D$ be the canonical map.
 \[
  \xymatrix
{
 0
 \ar@{->}[r]
  & Y
\ar@{->}[r]^{\phi}
\ar@{->}[d]^{\xi}
 & X
\ar@{->}[r]^{f}
\ar@{->}[d]^{\pi}
& V
\ar@{->}[r]
\ar@{->}[d]
&0
\\
 0
 \ar@{->}[r]
  &L
\ar@{->}[r]
 & X/D
\ar@{->}[r]^{\ov{f}}
& V/V'
    \ar@{->}[r]
    &0
\
 }
\]
We have $\ker \xi \subseteq D$. So as argued before we get that $\image \xi \in \A(\R)(1)$. Note $\coker \xi $ is a homorphic image of $V' = f(D)$. It follows that $L \in \A(\R)(2)$. As $V/V' \in \gmod(\R)$ it follows from \ref{closed} that $X/D \in \A(\R)$.
\end{proof}
\section{Coherent functors}
Coherent functors in the general setting of abelian categories were introduced by Auslander, \cite{A}. In our setting, coherent functors on the category of modules over a Noetherian commutative ring, an extensive treatment was made by Hartshorne, \cite{H}.
Let us first recall the notion of coherent functors. Let $A$ be a Noetherian ring and let $\smod(A)$ denote the category of finitely generated $A$-modules.
Let $\F$ denote the category of all covariant $A$-linear functors from $\smod(A)$ to $\smod(A)$. Morphisms in $\F$ are given by natural transformation of functors. It is easily proved that $\F$ is an abelian category. Let $U \in \smod(A)$. Set $h_U(-) = \Hom_A(U, -)$.
\begin{definition}
Let $F \in \F$.
  \begin{enumerate}
    \item $F$ is said to be representable if $F \cong h_U$ for some $U \in \smod(A)$.
    \item $F$ is said to be coherent if there exists an exact sequence of functors
    $$ h_U \rt h_V \rt F \rt 0.$$
  \end{enumerate}
\end{definition}
\begin{remark}
  A natural transformation of $A$-linear functors $h_U \rt h_V$ is induced by a unique $A$-linear map $V \rt U$, see \cite[1,2]{H}.
\end{remark}

\s \label{setup} Let $\R = \bigoplus_{n \geq 0} \R_n$ be a standard graded ring with $\R_0 = A$. Let $X = \bigoplus_{n \in \Z} X_n \in \Lmod(\R)$. Let $F\in \F$. Set $F(X) = \bigoplus_{n \in \Z} F(X_n)$. We note that if $U \in \smod(A)$ then $h_U(X)$ is an $\R$-module. Furthermore if $\theta \colon h_U \rt h_V$ is a natural transformation where $U, V \in \smod(A)$ and $\theta$ induced by $f \colon V \rt U$ then $\theta^*_X \colon h_U(X) \rt h_V(X)$ is $\R$-linear. It follows that if $F \in \F$ is coherent then $F(X)$ is an $\R$-module.

\s \label{fg} (with hypotheses as in \ref{setup}). Let $X \in \gmod(\R)$. Then note that $h_U(X) \in \gmod(X)$ for any $U \in \smod(A)$. It follows that if $F$ is coherent then $F(X) \in \gmod(\R)$.

The following is the main result of this section.
\begin{theorem}
\label{coh-ar}(with hypotheses as in \ref{setup}). Let $F \in \F$ be a coherent functor. If $X \in \A(\R)$ then $F(X) \in \A(\R)$.
\end{theorem}
\begin{proof}
Let $h_U \rt h_V \rt F \rt 0$ where $U, V \in \smod(A)$.
  We prove by induction that if $X \in \A(\R)(i)$ (for i = 0, 1,2) then $F(X) \in \A(\R)$. The case when $i = 0$ follows from \ref{fg}.

  Let $X \in \A(\R)(1)$. Then there exists a standard graded ring $\Sc$, an inclusion of graded rings $\R \rt \Sc$ with $\R_0 = \Sc_0 = A$ and $M \in \gmod(\R)$ and $N \in \gmod(\Sc)$ with an exact sequence
  \begin{equation*}
    0 \rt M \rt N \rt X \rt 0  \tag{*}
  \end{equation*}
   of $\R$-modules.  We apply $h_U(-)$ and $h_V(-)$ then we have the following two commutative diagrams with exact rows
   \[
  \xymatrix
{
 0
 \ar@{->}[r]
  & h_U(M)
\ar@{->}[r]
\ar@{->}[d]^{g_M}
 & h_U(N)
\ar@{->}[r]^{p}
\ar@{->}[d]^{g_N}
& W_U
\ar@{->}[r]
\ar@{->}[d]^{\xi}
&0
\\
 0
 \ar@{->}[r]
  &h_V(M)
\ar@{->}[r]
 & h_V(N)
\ar@{->}[r]
& W_V
    \ar@{->}[r]
    &0
\
 }
\]
The second commutative diagram is
\[
  \xymatrix
{
 0
 \ar@{->}[r]
  & W_U
\ar@{->}[r]
\ar@{->}[d]^{\xi}
 & h_U(X)
\ar@{->}[r]
\ar@{->}[d]^{g_X}
& D_U
\ar@{->}[r]
\ar@{->}[d]^{\phi}
&0
\\
 0
 \ar@{->}[r]
  &W_V
\ar@{->}[r]
 & h_V(X)
\ar@{->}[r]
& D_V
    \ar@{->}[r]
    &0
\
 }
\]
Here $D_U, D_V$  are $\R$-submodules of $\Ext_A^1(U, M)$ and $\Ext^1_A(V, M)$ and so are finitely generated $\R$-modules.\\
By the first commutative diagram, by applying the snake lemma we get an exact sequence of $\R$-modules
\[
F(M) \xrightarrow{\theta} F(N) \rt \coker(\xi) \rt 0.
\]
We notice that $F(M) \in \gmod(\R)$ and $F(N) \in \gmod(\Sc)$. As $\theta$ is $\R$-linear we have $\image(\theta) \in \gmod(\R)$.
It follows that $\coker(\xi) \in \A(\R)(1)$.

By the second commutative diagram, by applying snake Lemma we obtain an exact sequence of $\R$-modules
\[
\ker(\phi) \xrightarrow{\delta} \coker(\xi) \rt F(X) \rt \coker(\phi) \rt 0.
\]
We note that $\coker(\delta) \in \A(\R)(1)$ as $\ker(\phi) \in \gmod(\R)$. As $\coker(\phi) \in \gmod(\R)$ it follows that $F(X) \in \A(\R)(2)$.

Let $X \in \A(\R)(2)$. Then there exists an exact sequence
\[
0 \rt Y \rt X \rt D \rt 0
\]
where $Y  \in \A(\R)(1)$ and $D \in \gmod(\R)$.
Applying the functors $h_U, h_V$ we get a commutative diagram with exact rows
  \[
  \xymatrix
{
 0
 \ar@{->}[r]
  & h_U(Y)
\ar@{->}[r]
\ar@{->}[d]^{g_Y}
 & h_U(X)
\ar@{->}[r]
\ar@{->}[d]^{g_X}
& D_U
\ar@{->}[r]
\ar@{->}[d]^{\xi}
&0
\\
 0
 \ar@{->}[r]
  &h_V(Y)
\ar@{->}[r]
 & h_V(X)
\ar@{->}[r]
& D_V
    \ar@{->}[r]
    &0
\
 }
\]
Here $D_U, D_V$ are $\R$-submodules of $h_U(D),h_V(D)$ respectively and so are finitely generated. By snake Lemma we have an exact sequence of $\R$-modules
\[
\ker \xi \xrightarrow{\delta}  F(Y) \rt F(X) \rt \coker \xi \rt 0
\]
As $\ker \xi \in \gmod(\R)$ and as proved earlier $F(Y) \in \A(\R)$, it follows from \ref{q-fg} that $\coker \delta \in \A(\R)$. As $\coker \xi \in \gmod(\R)$ it follows from \ref{closed} that $F(X) \in \A(\R)$.
\end{proof}
\section{Quasi-finite modules}
Let $\R = \bigoplus_{n \geq 0}\R_n$ be a standard graded ring with $\R_0 = A$. Let $\R_+ = \bigoplus_{n \geq 1}\R_n$ be the irrelevant ideal of $\R$.  Define
$\Lmod_f(\R)$ to be the full subcategory of graded $\R$-modules $X$ with $X_n$ a finitely generated $A$-module for all $n \in \Z$ and $X_n = 0$ for all $n \ll 0$.
We say $X \in \Lmod_f(\R)$ is \emph{quasi-finite} if $H^0_{\R_+}(X)_n = 0$ for all $n \gg 0$. Clearly finitely generated graded $\R$-modules are quasi-finite. Quasi-finite modules were introduced in \cite{HPV}. They have been extensively investigated in \cite{KP}  in the multi-graded setting. We recall results from \cite{KP} and prove in the singly graded case where more transparent proofs is available.

We first prove
\begin{proposition}\label{mod}
Let $X \in \Lmod_f(\R)$. Then $H^0_{\R_+}(X/H^0_{\R_+}(X)) = 0$. Thus $X/H^0_{\R_+}(X)$ is quasi-finite.
\end{proposition}
\begin{proof}
Let $E = H^0_{\R_+}(X/H^0_{\R_+}X)$. Let $[x] \in E$ be homogeneous with $x \in X$ homogeneous. Then $\R_{+}^sx \in H^0_{\R_+}(X)$. Notice $\R_{+}^sx$ is a finitely generated $\R$-submodule of
$H^0_{\R_+}(X)$. It follows that $\R_{+}^l(\R_{+}^sx) = 0$. So $\R_+^{l+s}x = 0$. Thus $x \in H^0_{\R_+}(X)$. So $[x] = 0$. Thus $E = 0$.
The result follows.
\end{proof}
%   \section*{Acknowledgements}
Next we show
\begin{proposition}
\label{ass-qf}
  Let $X \in \Lmod_f(\R)$ be a quasi-finite $\R$-module. Then \\ $\Ass_A X_n \subseteq \Ass_A X_{n+1}$ for all $n \gg 0$.
\end{proposition}
\begin{proof}
Set $Y = X/H^0_{\R_+}(X)$. Then by \ref{mod} it follows that $H^0_{\R_+}(Y) = 0$. As $X$ is quasi-finite it follows that $Y_n = X_n$ for $n \gg 0$.
Let $\R_+ = (r_1, \ldots, r_s)$ where $r_i \in \R_1$. We have an obvious map
\begin{align*}
  \phi \colon Y(-1)& \rt Y^s  \\
  y &\mapsto (r_1y, \cdots, r_sy).
\end{align*}
Notice if $y \in \ker \phi$ then $y \in (0 \colon_Y \R_+) \subseteq H^0_{\R_+}(Y) = 0$. So $\ker \phi = 0$. As $\phi$ is graded we have an inclusion $Y_{n} \subseteq Y_{n+1}^s$ of $A$-modules.
Therefore $\Ass_A Y_n \subseteq \Ass_A Y_{n+1}$ for all $n$. The result follows.
\end{proof}
\s Let $X \in \Lmod_f(\R)$. Then we say $ X$ has \emph{stable associate primes} if there exists $n_0$ with $\Ass_A X_n = \Ass_A X_{n_0}$ for all
$n \geq n_0$. Next we show
\begin{proposition}\label{exact}
  Let $X \in \Lmod_f(\R)$. Set $Y = X/H^0_{\R_+}(X)$. Then $\Ass_A X = \Ass_A Y \cup \Ass_A H^0_{\R_+}(X)$. We have if $P \in \Ass_A Y_s$ then $P \in \Ass_A X_{s+r}$ for some $r \geq 1$.
\end{proposition}
\begin{proof}
Clearly we have
\[
\Ass_A H^0_{\R_+}(X) \subseteq \Ass_A  X \subseteq \Ass_A Y \cup \Ass_A H^0_{\R_+}(X).
\]
Let $P \in \Ass_A Y$. We localize at $P$ and so we may assume $A$ is local with maximal ideal $\m$ and $\m \in \Ass_A Y$.
Suppose $\m = (0 \colon [y]) $ with $[y] \in Y_s$. So $\m y \in H^0_{\R_+}(X)$. Thus ${\R_+}^l\m y = 0$ for some $l \geq 1$. As $[y] \neq 0$ in $Y$ it follows that $(\R_+)^ly \neq 0$. Let
$u \in (\R_+)^ly$ be non-zero and homogeneous. Then $\m u = 0$. So $\m \in \Ass_A X_{s+ r}$ for some $r \geq l$.
\end{proof}
The final result in this section is
\begin{theorem}\label{ass-stable}
  Let $X \in \Lmod_f(\R)$. Assume $\Ass_A X$ is finite. If $\Ass_A H^0_{\R_+}(X)_n$ is stable for $n \gg 0$ then $\Ass_A X_n$ is stable for $n \gg 0$.
\end{theorem}
\begin{proof}
  Set $Y = X/H^0_{\R_+}(X)$. By \ref{exact} it follows that $\Ass_A Y$ is finite. As $Y$ is quasi-finite (see \ref{mod}) it follows that $\Ass_A Y_n$ is stable, see \ref{ass-qf}.
  Assume that for $n \geq n_0$ we have $\Ass_A H^0_{\R_+}(X) = \{ P_1, \ldots, P_r \}$ and $\Ass_A Y_n = \{ Q_1, \ldots, Q_s \}$. We prove that for $n \geq n_0$ we have $\Ass_A X_n \subseteq \Ass_A X_{n +1}$. As $\Ass_A X$ is finite the result follows. Note $P_i \in \Ass_A X_n$ for all $n \geq n_0$. Suppose $Q \in \Ass_A X_n$ for some $n \geq n_0$ with $Q \neq P_i$ for all $i$. Note $Q \in \Ass_A Y_n$.
  We localize at $Q$.
  So we may assume $(A,\m)$ is local, $\m \in \Ass_A X_n $ but $\m \notin \Ass_A H^0_{\R_+}(X)_n$.
   Let $\R_+ = (r_1, \ldots, r_l)$ where $r_i \in \R_1$. If $E$ is a graded
  $\R$-module we have a map
\begin{align*}
  \phi_E \colon E(-1)& \rt E^l \\
  e &\mapsto (r_1e, \cdots, r_le).
\end{align*}
Set $D = H^0_{\R_+}(X)$.  We have an exact sequence $0 \rt D  \rt X \rt Y \rt 0$. We take its $l$-fold direct sum $0 \rt D^l \rt X^l  \rt Y^l \rt 0$.
We have a commutative diagram:
\[
  \xymatrix
{
 0
 \ar@{->}[r]
  & D
\ar@{->}[r]
\ar@{->}[d]^{\phi_D}
 & X
\ar@{->}[r]
\ar@{->}[d]^{\phi_X}
& Y
\ar@{->}[r]
\ar@{->}[d]^{\phi_Y}
&0
\\
 0
 \ar@{->}[r]
  &D^l
\ar@{->}[r]
 & X^l
\ar@{->}[r]
& Y^l
    \ar@{->}[r]
    &0
\
 }
\]
By proof of \ref{ass-qf} we get that $\phi_Y$ is injective. Taking $\Hom_A(k, -)$ and noting that $\m \notin \Ass_A D_j$ for all $j \geq n_0$ we obtain a commutative diagram
\[
  \xymatrix
{
 0
 \ar@{->}[r]
 & \Hom_A(k, X_n)
\ar@{->}[r]
\ar@{->}[d]^{\phi^*_X}
& \Hom_A(k, Y_n)
\ar@{->}[d]^{\phi^*_Y}
\\
 0
 \ar@{->}[r]
 & \Hom_A(k, X_{n+1}^l)
\ar@{->}[r]
& \Hom_A(k, Y_{n+1}^l)
\
 }
\]
As $\phi^*_Y$ is injective it follows that $\phi^*_X$ is also injective. Thus $\Hom_A(k, X_{n+1}) \neq 0$. So $\m \in \Ass_A X_{n+1}$. The result follows.
\end{proof}

\section{An analogue of a result by Amao}
\s \label{setup-amao}
Let $A$ be a Noetherian ring.
Let $\R$ and $\Sc$ be standard graded $A$-algebras wtih an inclusion of graded rings $\R \subseteq \Sc$ and $\R_0 = \Sc_0  = A$.
Let $M$ be a finitely generated graded $\R$-module and let $N$ be a finitely generated $\Sc$-module. Assume we have a inclusion $M \subseteq N$ of graded $\R$-modules.
In this section we prove the following:
\begin{theorem}\label{amao}
(with hypotheses as in \ref{setup-amao} The following results hold:
\begin{enumerate}[\rm (1)]
  \item $\Ass_A N/M$ is finite.
  \item If $A$ is local and $\ell(N_n/M_n) < \infty$ for all $n$ then the function $n \rt \ell(N_n/M_n)$ is of polynomial type and of degree $\leq \dim \Sc - 1$.
\end{enumerate}
\end{theorem}
A multigraded version of part(1) of the above result was proved in \cite{HF}. The proof of our singly graded version is considerably simpler and also gives (2).

\s \label{const} \emph{Construction}: One of the chief problems in trying to prove the above theorem is that $N/M$ is not a
$\Sc$-module. It is also, in general, not a finitely generated $\R$-module. We generalize Amao's technique, \cite{Am}, in the case $M = \R$ and $N = \Sc$,
is to consider the Rees ring $R(\R_+\Sc, \Sc)$. We generalize this technique as follows:

Let $M_{\geq r} = \bigoplus_{n \geq r} M_n$. Then for $r \gg 0$ we have that $M_{\geq r} $ is generated in degree $r$ as a $\R$-module.
Similarly $N_{\geq r}$ is generated in degree $r$ as a $\Sc$-module (for $r \gg 0$).  We choose $r \gg 0$ such that both hold. Clearly we may replace $M, N$ by
$M_{\geq r}$ and $N_{\geq r}$ respectively. We then shift the modules by $r$. Thus we may assume $M = \bigoplus_{n \geq 0}M_n$ and $N = \bigoplus_{n \geq 0} N_n$ and that $M$( $N$) is generated by
$M_0$($N_0$) as a $\R$($\Sc$)-module.

Consider the graded trivial extension $\wt{\R} = \R \ltimes M[-1]$ with $\wt{\R}_n = (\R_{n}, M_{n-1})$. It is readily checked that $\wt{\R}$ is standard graded with $\wt{\R}_0 = A$.
Similarly we consider the trivial extension $\wt{\Sc} = \Sc \ltimes N[-1]$. Note we have an inclusion of standard graded rings $\wt{\R} \subseteq \wt{\Sc} $ with $\wt{\R}_0  = \wt{\Sc}_0 = A$.

Next we consider $T = R(\wt{\R}_+ \wt{\Sc}, \wt{\Sc})$ the Rees algebra of $\wt{\R}_+ \wt{\Sc}$ in $\wt{\Sc}$. We note that $T$ is canonically bi-graded with $T_{ij} = \wt{\R}_i \wt{\Sc}_{j-i}$.
It is easy to check that $T_{i +1,j} \subseteq T_{ij}$. We note that $T_{ij}  =0$ if $j < i$. Consider the associated graded ring $G = \text{gr}_{\wt{\R}_1\wt{\Sc}}\Sc$. We note that $G_{ij} = T_{ij}/T_{i+1,j}$.
We now give:
\begin{proof}[Proof of Theorem \ref{amao}] We make the construction as in \ref{const}. First consider the ideal $L = \bigoplus_{j > i}G_{ij}$ of $G$. Next consider the ideal $L^*$ of $G$ consisting of second components  of $L$.

We now consider $T$, $G$ and $L^*$ singly graded by summing across columns, ie., $T = \bigoplus_{j \geq 0}T_j$ where $T_j = \bigoplus_{i \geq 0}T_{ij} = \bigoplus_{i = 0}^{j} T_{ij}$.
Similarly $G = \bigoplus_{j \geq 0}G_j$. We note that with this grading $T$ (and hence $G$) is standard graded.

We have $L^* = \bigoplus_{j \geq 0}L_j^*$.
Also $L_j^* = \bigoplus_{i = 0}^jL^*_{ij}$. We have
\[
L^*_{ij} = \left(0,  \frac{\R_iN_{j-i-1} + \Sc_{j-i}M_{i-1}}{ \R_{i+1}N_{j-i-2} + \Sc_{j-i-1}M_i } \right).
\]
We note that
\[
L^*_{0j} = \left(0,  \frac{N_{j-1} }{ \R_{1}N_{j-2} + \Sc_{j-1}M_0} \right) \quad \text{and} \ L^*_{j-1, j} = \left(0,  \frac{\R_{j-1}N_{0} + \Sc_{1}M_{j-2}}{ M_{j-1} } \right).
\]
We observe that the second components of $L^*_{j}$ give a filtration of $N_{j-1}/M_{j-1}$, where
\begin{align*}
  M_{j-1} &\subseteq R_{j-1}N_0 + \Sc_1 M_{j-2} \\
   &\subseteq \R_{j-2}N_1 + \Sc_2M_{j-3} \\
   &\subseteq \cdots \\
   &\subseteq \R_{j-i}N_{i-1} + \Sc_{i} M_{j - i - 1}\\
   &\subseteq \cdots \\
   &\subseteq \R_1N_{j-2} + \Sc_{j-1}M_0 \\
   &\subseteq N_{j-1}.
\end{align*}
It follows that $\Ass_A N_{j-1}/M_{j-1} \subseteq \Ass_A L^*_j$ and if $\ell_A(N_{j-1}/M_{j-1})$ is finite then $\ell_A(L_j^*) = \ell(N_{j-1}/M_{j-1})$.

As $G$ is a Noetherian ring and $L^*$ is an ideal of $G$ it follows from \ref{ex-m}, that $\Ass_A L^*$ is a finite set. So $\Ass_A N/M$ is also finite set.
Furthermore if $N_{j-1}/M_{j-1}$ has finite length for all $j$ then so does $L_j^*$. It follows that the function $j \rt \ell(L_j^*)$ is of polynomial type of degree at most $\dim G$. By
\cite[4.5.6]{BH} we have
$\dim G = \dim \wt{\Sc}$. But clearly $\dim \wt{\Sc} = \dim \Sc$ (this is Exercise 3.3.22 in \cite{BH}). The result follows.
\end{proof}
\section{Main Theorem}
In this section we prove Theorem \ref{main}. We restate it here for the convenience of the reader.
\begin{theorem}\label{main-body}
Let $A$ be a Noetherian ring and let $\R = \bigoplus_{n \geq 0}\R_n$ be a standard graded ring with $\R_0 = A$. Let $X  = \bigoplus_{n \in \Z} X_n\in \A(\R)$. Then
\begin{enumerate}[\rm (1)]
  \item  $X_i$ is finitely generated $A$-module for all $i \in \Z$ and $X_i = 0$ for $i \ll 0$.
  \item There exists $n_0$ such that $\Ass_A X_n = \Ass_A X_{n_0}$ for all $n \geq n_0$.
   \item If $A$ is local and $X_n$ has finite length  as an $A$-module for all $n$ then there exists $P_X(z) \in \mathbb{Q}[z]$ such that $P_X(n) = \ell_A(X_n)$ for all $n \gg 0$.
  \item If $F$ is a coherent functor on the category of finitely generated $A$-modules then $F(X) = \bigoplus_{n \in \Z} F(X_n)\in \A(\R)$.
\item For an ideal $J$ in $A$, there exists $c_J$ such that \\ $\grade(J, X_n) = \grade(J, X_{c_J})$ for all $n \geq c_J$.
\end{enumerate}
\end{theorem}

We will need the following exercise problem from \cite[6.7]{Ma}.
\begin{proposition}
  \label{ex-m} Let $R\rt S$ be a ring homorphism of Noetherian rings and let
  $M$ be a finitely generated $S$-module. Then $\Ass_R M = (\Ass_S M ) \cap R$. In particular $\Ass_A M$ is finite.
\end{proposition}

We now give:
\begin{proof}[Proof of Theorem \ref{main}]
(1) This follows from the construction of the category $\A(\R)$.

(2) We prove step by step  when $X \in \A(\R)(i)$. When $i = 0$ we have that $X$ is a finitely generated graded $\R$-module. From Proposition \ref{ex-m} it follows that $\Ass_A X$ is a finite set. As $H^0_{\R_+}(X)_n = 0$ for $n \gg 0$ it follows from \ref{ass-qf} that $\Ass_A X_n \subseteq \Ass_A X_{n+1}$ for $n \gg 0$. The result follows.

When $i = 1$ then there exists a standard graded ring $\Sc$, an inclusion of graded rings $\R \subseteq \Sc$ with $\R_0 = \Sc_0 = A$ and a finitely generated graded $\R$-module $M$ and a finitely generated graded $\Sc$-module $N$ with an inclusion of  graded $\R$-modules $ M \subseteq  N$ with $X = N/M$. By Theorem \ref{amao}, $\Ass_A X$ is a finite set. By the exact sequence
$0 \rt M \rt N \rt X \rt 0$ it follows that $H^0_{\R_+} (X)_n =  H^0_{\R_+}(\Sc)_n$ for $n \gg 0$. We note that
$ E = H^0_{\R_+}(\Sc) =  H^0_{\R_+\Sc}(\Sc)$ is an $\Sc$-ideal. So $\Ass_A E$ is a finite set by \ref{ex-m}. Let $\Sc_1 = (s_1, \ldots, s_l)$. Consider the map
\begin{align*}
  \phi \colon  &\Sc \rt  \Sc(+1)^l\\
  s&\mapsto (ss_1,\ldots, ss_l).
\end{align*}
We note that $\ker \phi = (0 \colon_\Sc \Sc_+)$. So $(\ker \phi)_n = 0$ for $n \gg 0$. Thus $\phi$ induces inclusion $\Sc_{n} \rt \Sc_{n+1}^l$ for all $n \gg 0$. Note that $\phi$ induces an inclusion $E_n \rt E_{n + 1}^{l}$ for all $n \gg 0$. As this map is $A$-linear we get  $\Ass_A E_n \subseteq \Ass_A E_{n+1}$ for all $n \gg 0$. As $\Ass_A E$ is finite we get that $\Ass_A E_n$ is stable for $n \gg 0$. It follows that $\Ass_A H^0_{\R_+}(X)_n$ is stable for $n \gg 0$. As $\Ass_A X$ is a finite set it follows from Theorem \ref{ass-stable} that $\Ass_A X_n$ is stable for $n \gg 0$.

When $i = 2$ we have an exact sequence $0 \rt Y \rt X \rt D \rt 0$ where $Y \in \A(\R)(1)$ and $D$ is a finitely generated graded $\R$-module. We have
$\Ass_A X \subseteq \Ass_A Y \cup \Ass_A D$ is a finite set. We also have $H^0_{\R_+}(X)_n = H^0_{\R_+}(Y)_n$ for $n \gg 0$. By earlier assertion we have that $\Ass_A H^0_{\R_+}(Y)_n$ is stable for $n \gg 0$. So $\Ass_A H^0_{\R_+}(X)_n $ is stable for $n \gg 0$. The result follows from Theorem \ref{ass-stable}.

(3) We prove step by step  when $X \in \A(\R)(i)$. When $i = 0$ we have that $X$ is a finitely generated graded $\R$-module. As $\ell_A(X_n)$ is finite for all $n$, it follows that $\R/\ann_\R X$ is a standard graded ring over an Artin local ring. The result follows from \cite[4.1.3]{BH}.

When $i = 1$ then there exists a standard graded ring $\Sc$, an inclusion of graded rings $\R \subseteq \Sc$ with $\R_0 = \Sc_0 = A$ and a finitely generated graded $\R$-module $M$ and a finitely generated graded $\Sc$-module $N$ with an inclusion of  graded $\R$-modules $ M \subseteq  N$ with $X = N/M$.  As $\ell(N_n/M_n)$ is finite for all $n$ it follows from Theorem \ref{amao}, that the function
$n \rt \ell(N_n/M_n)$ is of polynomial type.

When $i = 2$ we have an exact sequence $0 \rt Y \rt X \rt D  \rt 0$ where $Y \in \A(\R)(1)$ and $D$ is a finitely generated graded $\R$-module. As $\ell(X_n)$ is finite for all $n$ it follows that $\ell(Y_n)$ and $\ell(D_n)$ is finite. As argued earlier the functions $n \rt \ell(Y_n)$ and $n \rt \ell(D_n)$ is of polynomial type.
It follows that $ n \rt \ell(X_n) =\ell(Y_n) + \ell(D_n)$ is of polynomial type.

(4) This result follows from Theorem \ref{coh-ar}.

(5) Recall if $E$ is a finitely generated $A$-module then $\grade(J, E) = \infty$ if and only if $E = JE$; equivalently $E \otimes A/J = 0$. The functor $F(-) = (-)\otimes A/J$ is coherent.
So $Y = X \otimes_A A/J \in  \A(\R)$. So by (2) we have $\Ass_A Y_n$ is stable. In particular either $\Ass Y_n = \emptyset $ or $\Ass_A Y_n \neq \emptyset $ for $n \gg 0$. Thus either $\grade(J, X_n)$ is finite for all $n \gg 0$ OR $\grade(J, X_n) = \infty $ for all $n \gg 0$.

Now consider the case when $\grade(J, X_n) < \infty $ for all $n \gg 0$. If $J = (a_1, \ldots, a_m)$ then note that for any finitely generated $A$-module $E$ with $E \neq JM$ we have $\grade(J, E) \leq m$.  The functors $G_i(-) = \Ext^i_A(A/J, -)$ for $i = 0, \ldots, m$ are coherent. Set $Y^i = G_i(X) \in \A(\R)$ for $i = 0, \ldots, m$. By (2) we have
 we have $\Ass_A Y^i_n$ is stable. In particular either $\Ass Y^i_n = \emptyset $ or $\Ass_A Y^i_n \neq \emptyset $ for $n \gg 0$.
 It follows that there exists $r \leq m-1$ with

 (i)  $\Ass Y^i_n = \emptyset $ for all $n \gg 0$ and for all $i < r$.

 (ii) $\Ass  Y^r_n \neq \emptyset$ for all $n \gg 0$.

 Thus $\grade(J, X_n) = r$ for all $n \gg 0$.
\end{proof}

Next we give an affirmative answer to questions raised in \ref{question}.
\begin{corollary}
  Let $A$ be a Noetherian ring and let $\R = \bigoplus_{n \geq 0}\R_n$ be a standard graded ring with $\R_0 = A$.  Assume there exists a standard graded ring $\Sc$, an inclusion of graded rings $\R \subseteq \Sc$ with $\R_0 = \Sc_0 = A$ and a finitely generated graded $\R$-module $M$ and a finitely generated graded $\Sc$-module $N$ with an inclusion of  graded $\R$-modules $ M \subseteq  N$.
  \begin{enumerate}[\rm(1)]
\item
$\Ass_A F(N_n/M_n)$ is stable for $n \gg 0$ ?
\item
if $J$ is an ideal in $A$ then  $\grade(J, F(N_n/M_n))$ is constant for $n \gg 0$.
\item
if $A$ is local and  $F(N_n/M_n)$ has finite length for all $n$ then
the function  $n \rt \ell(F(N_n/M_n))$  of polynomial type.
\end{enumerate}
\end{corollary}
\begin{proof}
  The result follows from Theorem \ref{main-body} as $N/M \in \A(\R)$.
\end{proof}


\begin{thebibliography} {99}

\bibitem{Am}
J.~O.~Amao,
\emph{On a certain Hilbert polynomial},
J. London Math. Soc. (2), 14, (1976), 13--20.

\bibitem{A}
M.~Auslander,
\emph{Coherent functors},
 in "Proc. Conf. Categorical Algebra," La Jolla 1965, pp. 189--231, Springer-Verlag, 1966.

\bibitem{BM}
A.~Banda and L.~Melkersson,
\emph{Coherent functors and asymptotic stability},
J. Algebra, 522, (2019), 1–10.


\bibitem{Br}
M.~Brodmann,
\emph{Asymptotic stability of $Ass(M/I^nM)$},
 Proc. Amer. Math. Soc. 74,  (1979),  16--18

\bibitem{Br2}
\bysame,
\emph{The asymptotic nature of the analytic spread},
 Math. Proc. Cambridge Philos. Soc. 86, (1979), 35--39.


\bibitem {BH}  W. Bruns and J. Herzog,
\emph{Cohen-Macaulay Rings}, revised edition,
Cambridge Studies in Advanced Mathematics, 39.
Cambridge University Press, 1998.

%\bibitem{Bu}
%L.~Burch,
%\emph{Codimension and analytic spread},
%Math. Proc. Cambridge Philos. Soc. 72, (1972), 369--373.

\bibitem{H}
R.~Hartshorne,
\emph{Coherent functors},
Adv. Math. 140, (1998), 44--94.

\bibitem{HF}
F.~Hayasaka,
\emph{Asymptotic stability of primes associated to homogeneous components of multigraded modules},
J. Algebra, 306, (2006), 535–-543.

\bibitem{HPV}
J.~Herzog, T.~J.~Puthenpurakal and J.~K.~Verma,
\emph{Hilbert polynomials and powers of ideals},
Math. Proc. Cambridge Philos. Soc., 145, (2008),  623--642.

\bibitem{KN}
D.~Katz and C.~Naud\'{e},
\emph{Prime ideals associated to symmetric powers of a module},
Comm. Algebra, 23, (1995),  4549–4555.

\bibitem{KW}
D.~Katz and E.~West,
\emph{A linear function associated to asymptotic prime divisors},
Proc. Amer. Math. Soc.,  132,   (2003), 1589--1597.

\bibitem{KP}
D.~Katz and T.~J.~Puthenpurakal,
\emph{Quasi-finite modules and asymptotic prime divisors},
J. Algebra, 380, (2013), 18--29.

\bibitem{KS}
A.~K.~Kingsbury and R.~Y.~Sharp,
\emph{Asymptotic behavior of certain sets of prime ideals},
 Proc. Amer. Math. Soc. 124, (1996), 1703--1711

\bibitem{KV}
V.~Kodiyalam,
\emph{Homological invariants of powers of an ideal},
Proc. Amer. Math. Soc.118, (1993),  757--764

 \bibitem{Ma}
H.~Matsumura, 
\emph{Commutative ring theory}, Cambridge Studies in Advanced
  Mathematics, vol.~8, Cambridge University Press, Cambridge, 1986.



\bibitem{ME}
S.~McAdam and P.~Eakin,
\emph{The asymptotic Ass},
J. Algebra 61, (1979), 71--81

\bibitem{Mc}
S.~McAdam,
\emph{Asymptotic Prime Divisors},
 Lecture Notes in Mathematics, vol. 1023, Springer-Verlag, 1983.

\bibitem{MS}
L.~Melkersson and P.~Schenzel,
\emph{Asymptotic prime ideals related to derived functors},
Proc. Amer. Math. Soc. 117, (1993) 935--938.

\bibitem{R}
 L.~J.~Ratliff~Jr,
 \emph{On prime divisors of $I^n$, $n$ large},
  Michigan Math. J. 23, (1976), 337–-352


\bibitem{Se}
T.~Se,
\emph{Covariant functors and asymptotic stability},
J. Algebra, 484, (2017), 247–264.

\bibitem{T}
E.~Theodorescu,
\emph{Derived functors and Hilbert polynomials},
Math. Proc. Cambridge Philos. Soc.132, (2002),  75- –88.

\bibitem{W}
E.~West,
\emph{Primes associated to multigraded modules},
J. Algebra, 271,  (2004), 427--453.

\end{thebibliography}
\end{document}